\documentclass[11pt]{article}

\usepackage{amssymb,amsmath,amsthm}
\usepackage{enumerate}
\usepackage{cases}

\newcommand{\n}{\noindent}
\newcommand{\ovl}{\overline}
\newcommand{\intl}{\int\limits}

\newcommand{\bb}[1]{\mathbb{#1}}
\newcommand{\cl}[1]{\mathcal{#1}}
\newcommand{\vp}{\varepsilon}

\newcommand{\ds}{\displaystyle}

\theoremstyle{plain}
\newtheorem{thm}{Theorem}[section]
\newtheorem{lem}{Lemma}[section]
\newtheorem{pro}{Proposition}[section]
\newtheorem{cor}{Corollary}[section]

\theoremstyle{remark}

\theoremstyle{definition}

\numberwithin{equation}{section}

\title{Isolated Singularities of Polyharmonic Inequalities}

\author{Marius Ghergu\footnote{School of Mathematical Sciences,
    University College Dublin, Belfield, Dublin 4, Ireland; {\tt
      marius.ghergu@ucd.ie}}, Amir Moradifam\footnote{Dept.~of
  Mathematics, University of Toronto, Toronto, Ontario, CANADA M5S 2E4;
  {\tt amir@math.utoronto.ca}}, and Steven
  D.~Taliaferro\footnote{Mathematics Department, Texas A\&M
    University, College Station, TX 77843-3368; {\tt
      stalia@math.tamu.edu}} 
\footnote{Corresponding author, Phone 001-979-845-2404, Fax 001-979-845-6028}}

\date{}
 
\begin{document}
\date{} 
\maketitle

\thispagestyle{empty}

\begin{abstract}
We study nonnegative classical solutions $u$ of the polyharmonic inequality
\[
 -\Delta^mu \ge 0 \quad \text{in}\quad B_1(0) - \{0\} \subset{\bb R}^n.
\]
We give necessary and sufficient conditions on integers $n\ge 2$ and
$m\ge 1$ such that these solutions $u$ satisfy a pointwise a priori
bound as $x\to 0$. In this case we show that the optimal bound for $u$
is
\[
 u(x) = O(\Gamma(x))\quad \text{as}\quad x\to 0
\]
where $\Gamma$ is the fundamental solution of $-\Delta$ in ${\bb R}^n$.
\medskip

\n {\it Keywords:} Polyharmonic inequality, isolated singularity.

\n 2010 Mathematics Subject Classification Codes: 35B09, 35B40,
35B45, 35C15, 35G05, 35R45.
\end{abstract}

\section{Introduction}\label{sec1}

\indent 

It is easy to show that there does not exist a pointwise a priori
bound as $x\to 0$ for $C^2$ nonnegative solutions $u(x)$ of
\begin{equation}\label{eq1.1}
 -\Delta u\ge 0 \quad \text{in}\quad B_1(0)-\{0\} \subset {\bb R}^n,
\quad n\ge 2.
\end{equation}
That is, given any continuous function $\psi\colon (0,1)\to
(0,\infty)$ there exists a $C^2$ nonnegative solution $u(x)$ of
\eqref{eq1.1} such that
\[
 u(x)\ne O(\psi(|x|)) \quad \text{as}\quad x\to 0.
\]
The same is true if the inequality in \eqref{eq1.1} is reversed.

In this paper we study $C^{2m}$ nonnegative solutions of the polyharmonic
inequality
\begin{equation}\label{eq1.2}
 -\Delta^m u\ge 0\quad \text{in}\quad B_1(0) - \{0\}\subset {\bb R}^n
\end{equation}
where $n\ge 2$ and $m\ge 1$ are integers. We obtain the following result.

\begin{thm}\label{thm1.1}
  A necessary and sufficient condition on integers $n\ge 2$ and $m\ge
  1$ such that $C^{2m}$ nonnegative solutions $u(x)$ of \eqref{eq1.2}
  satisfy a pointwise a priori bound as $x\to 0$ is that 
\begin{equation}\label{eq1.3}
\text{either $m$ is even or $n<2m$.}
\end{equation}
In this case, the optimal bound for $u$ is
\begin{equation}\label{eq1.5}
 u(x) = O(\Gamma_0(x))\quad \text{as}\quad x\to 0,
\end{equation}
where
\begin{equation}\label{eq1.6}
 \Gamma_0(x)= \begin{cases}
              |x|^{2-n}&\text{if $n\ge 3$}\\
\noalign{\medskip}
\log \dfrac5{|x|}&\text{if $n=2$.}
             \end{cases}
\end{equation}
\end{thm}

The $m$-Kelvin transform of a function $u(x)$, $x\in\Omega\subset 
{\bb R}^n-\{0\}$, is defined by 
\begin{equation}\label{eq1.5.1}
v(y)=|x|^{n-2m}u(x) \qquad \text{where} \quad x=y/|y|^2. 
\end{equation}
By direct computation, $v(y)$ satisfies 
\begin{equation}\label{eq1.5.2}
\Delta^mv(y)=|x|^{n+2m}\Delta^mu(x). 
\end{equation}
See \cite[p.~221]{WX} or \cite[p.~660]{X}. This fact and
Theorem \ref{thm1.1} immediately imply the following result.

\begin{thm}\label{thm1.2}
  A necessary and sufficient condition on integers $n\ge 2$ and $m\ge
  1$ such that $C^{2m}$ nonnegative solutions $v(y)$ of
\[
 -\Delta^m v\ge 0\quad \text{in}\quad {\bb R}^n-B_1(0)
\]
satisfy a pointwise a priori bound as $|y|\to \infty$ is that 
\eqref{eq1.3} holds. In this case, the optimal bound
for $v$ is
\begin{equation}\label{eq1.7}
 v(y) = O(\Gamma_\infty(y))\quad \text{as}\quad |y|\to\infty
\end{equation}
where
\begin{equation}\label{eq1.71}
 \Gamma_\infty(y) = \begin{cases}
                     |y|^{2m-2}&\text{if $n\ge 3$}\\
|y|^{2m-2}\log(5|y|)&\text{if $n=2$.}
                    \end{cases}
\end{equation}
\end{thm}

The estimates \eqref{eq1.5} and \eqref{eq1.7} are optimal because
$\Delta^m\Gamma_0=0=\Delta^m\Gamma_\infty$ in ${\bb R}^n-\{0\}$.

The sufficiency of condition \eqref{eq1.3} in Theorem \ref{thm1.1} and
the estimate \eqref{eq1.5} are an immediate consequence of the
following theorem, which gives for $C^{2m}$ nonnegative solutions $u$
of \eqref{eq1.2} one sided estimates for $\Delta^\sigma u$,
$\sigma=0,1,2,\dots, m$, and estimates for $|D^\beta u|$ for certain 
multi-indices $\beta$.

\begin{thm}\label{thm1.3}
Let $u(x)$ be a $C^{2m}$ nonnegative solution of
\begin{equation}\label{eq4.1}
 -\Delta^mu\ge 0\quad \text{in}\quad B_2(0) -\{0\}\subset {\bb R}^n,
\end{equation}
where $n\ge 2$ and $m\ge 1$ are integers. Then for each nonnegative
integer $\sigma\le m$ we have
\begin{equation}\label{eq4.2}
(-1)^{m+\sigma} \Delta^\sigma u(x) 
\le C\left|\frac{d^{2\sigma}}{d|x|^{2\sigma}} \Gamma_0(|x|)\right|
\quad \text{for}\quad 0<|x|<1
\end{equation}
where $\Gamma_0$ is given by \eqref{eq1.6} and $C$ is a positive
constant independent of $x$.

Moreover, if $n<2m$ and $\beta$ is a multi-index then
\begin{equation}\label{eq4.3}
 |D^\beta u(x)| = O\left(\left|\frac{d^{|\beta|}}{d|x|^{|\beta|}} 
\Gamma_0(|x|)\right|\right) \quad \text{as}\quad x\to 0
\end{equation}
for 
\begin{equation}\label{neq4.4}
 |\beta|\le \begin{cases}
             2m-n&\text{if $n$ is odd}\\
2m-n-1&\text{if $n$ is even.}
            \end{cases}
\end{equation}
\end{thm}

There is a similar result when the singularity is at infinity.

\begin{thm}\label{thm1.4}
Let $v(y)$ be a $C^{2m}$ nonnegative solution of
\begin{equation}\label{neq4.1}
 -\Delta^mv\ge 0\quad \text{in}\quad {\bb R}^n - B_{1/2}(0),
\end{equation}
where $n\ge 2$ and $m\ge 1$ are integers. Then for each nonnegative
integer $\sigma\le m$ we have
\begin{equation}\label{neq4.2}
(-1)^{m+\sigma} \Delta^\sigma(|y|^{2\sigma-2m}v(y)) 
\le C \begin{cases}
|y|^{-2}\log 5|y| &\text{if $\sigma=0$ and $n=2$}\\
|y|^{-2}  &\text{if $\sigma\ge 1$ or $n\ge 3$}
\end{cases} 
\quad \text{for} \quad |y|>1
\end{equation}
where $C$ is a positive
constant independent of $y$.

Moreover, if $n<2m$ and $\beta$ is a multi-index satisfying 
\eqref{neq4.4} then
\begin{equation}\label{neq4.3}
 |D^\beta v(y)| = O\left(\left|\frac{d^{|\beta|}}{d|y|^{|\beta|}} 
\Gamma_\infty(|y|)\right|\right) \quad \text{as}\quad |y|\to \infty
\end{equation}
where $\Gamma_\infty$ is given by \eqref{eq1.71}.
\end{thm}

Note that in Theorems \ref{thm1.3} and \ref{thm1.4} we do not require that
$m$ and $n$ satisfy \eqref{eq1.3}.

Inequality \eqref{neq4.2} gives one sided estimates for
$\Delta^\sigma(|y|^{2\sigma-2m}v(y))$. Sometimes one sided estimates
for $\Delta^\sigma v$ also hold. For example, in the important
case $m=2$, $n=2$ or 3, and the singularity is at the infinity, we
have the following corollary of Theorem \ref{thm1.4}.

\begin{cor}\label{cor1.1}
Let $v(y)$ be a $C^4$ nonnegative solution of 
\[
-\Delta^2 v \ge 0 \quad \text{in} \quad {\bb R}^n - B_{1/2}(0)
\]
where $n=2$ or $3$. Then 
\begin{equation}\label{eq1.14.1}
v(y)=O\left(\Gamma_\infty(|y|)\right) \quad\text{and} \quad 
|\nabla v(y)|=O\left(\left|\frac{d}{d|y|} 
\Gamma_\infty(|y|)\right|\right) \quad \text{as} \quad |y|\to\infty 
\end{equation}
and
\begin{equation}\label{eq1.14.2}
-\Delta v(y)<C\left|\frac{d^2}{d|y|^2} 
\Gamma_\infty(|y|)\right|\quad \text{for} \quad |y|>1
\end{equation}
where $\Gamma_\infty$ is given by \eqref{eq1.71} and $C$ is  a 
positive constant independent of $y$.
\end{cor}

The proof of Theorem~\ref{thm1.3} relies heavily on a representation
formula for $C^{2m}$ nonnegative solutions $u$ of \eqref{eq1.2}, which
we state and prove in Section~\ref{sec3}. This formula, which is valid
for all integers $n\ge 2$ and $m\ge 1$ and which when $m=1$ is
essentially a result of Brezis and Lions \cite{BL}, may also be useful
for studying nonnegative solutions in a punctured neighborhood of the
origin---or near $x=\infty$ via the $m$-Kelvin transform---of problems of
the form
\begin{equation}\label{eq1.8}
 -\Delta^m u = f(x,u)\quad\text{or}\quad 0\le - \Delta^mu \le f(x,u)
\end{equation}
when $f$ is a nonnegative function and $m$ and $n$ may or may not
satisfy \eqref{eq1.3}. Examples of such problems can
be found in \cite{CMS,CX,GL,H,MR,WX,X} and elsewhere.

Pointwise estimates at $x=\infty$ of solutions $u$ of problems
\eqref{eq1.8} can be crucial for proving existence results for entire
solutions of \eqref{eq1.8} which in turn can be used to obtain, via
scaling methods, existence and estimates of solutions of boundary
value problems associated with \eqref{eq1.8}, see
e.g. \cite{RW1,RW2}. An excellent reference for polyharmonic boundary
value problems is \cite{GGS}.

Lastly, weak solutions of $\Delta^mu = \mu$, where
$\mu$ is a measure on a subset of ${\bb R}^n$, have been studied in
\cite{CDM} and \cite{FKM}, and removable isolated singularities of 
$\Delta^mu=0$ have been studied in \cite{H}.

\section{Preliminary results}\label{sec2}

\indent 

In this section we state and prove four lemmas. Lemmas~\ref{lem2.1},
\ref{lem2.2}, and \ref{lem2.3} will only be used to prove
Lemma~\ref{lem2.4}, which in turn will be used in Section~\ref{sec3}
to prove Theorem~\ref{thm3.1}.

Lemmas \ref{lem2.1} and \ref{lem2.2} are well-known. We include their
very short proofs for the convenience of the reader.

\begin{lem}\label{lem2.1}
  Let $f\colon (0,r_2]\to [0,\infty)$ be a continuous function where
  $r_2$ is a finite positive constant. Suppose $n\ge 2$ is an integer
  and the equation
\begin{equation}\label{eq2.1}
v'' + \frac{n-1}r v' = -f(r)\qquad 0<r<r_2
\end{equation}
has a nonnegative solution $v(r)$. Then
\begin{equation}\label{eq2.2}
 \int^{r_2}_0 r^{n-1} f(r)\,dr < \infty.
\end{equation}
\end{lem}

\begin{proof}
Let $r_1=r_2/2$. Integrating \eqref{eq2.1} we obtain
\begin{equation}\label{eq2.3}
r^{n-1}v'(r) = r^{n-1}_1 v'(r_1) + \int^{r_1}_r \rho^{n-1}f(\rho)\,d\rho 
\quad \text{for}\quad 0<r<r_1.
\end{equation}
Suppose for contradiction that
\[
 r^{n-1}_1v'(r_1) + \int^{r_1}_{r_0} \rho^{n-1}f(\rho)\,d\rho \ge 1
\quad \text{for some} \quad r_0\in (0,r_1).
\]
Then for $0<r<r_0$ we have by \eqref{eq2.3} that
\[
 v(r_0) - v(r) \ge \int^{r_0}_r \rho^{1-n}\,d\rho \to \infty
\quad \text{as}\quad r\to 0^+
\]
which contradicts the nonnegativity of $v(r)$.
\end{proof}

\begin{lem}\label{lem2.2}
  Suppose $f\colon (0,R]\to {\bb R}$ is a continuous function, $n\ge
  2$ is an integer, and
\begin{equation}\label{eq2.4}
\int^R_0 \rho^{n-1}|f(\rho)|\,d\rho < \infty.
\end{equation}
Define $u_0\colon (0,R]\to {\bb R}$ by
\[
 u_0(r)= \begin{cases}
          \dfrac1{n-2} \left[\dfrac1{r^{n-2}} 
{\ds\int^r_0} \rho^{n-1}f(\rho)\,d\rho 
+ {\ds\int^R_r} \rho f(\rho)\,d\rho\right]&\text{if $n\ge 3$}\\
\noalign{\medskip}
\left(\log \dfrac{2R}r\right) {\ds\int^r_0} \rho f(\rho)\,d\rho 
+ {\ds\int^R_r} \rho\left(\log \dfrac{2R}\rho\right) f(\rho)\,d\rho
&\text{if $n=2$}.
         \end{cases}
\]
Then $u = u_0(r)$ is a $C^2$ solution of
\begin{equation}\label{eq2.5}
 -(\Delta u)(r) := -\left(u''(r) + \frac{n-1}r u'(r)\right) 
= f(r)\quad \text{for}\quad 0<r\le R.
\end{equation}
Moreover, all solutions $u(r)$ of \eqref{eq2.5} are such that
\begin{equation}\label{eq2.6}
 \int^r_0 \rho^{n-1}|u(\rho)|\,d\rho = 
\begin{cases}
O(r^2)&\text{as $r\to 0^+$ if $n\ge 3$}\\
\noalign{\medskip}
O\left(r^2 \log \dfrac1r\right)&\text{as $r\to 0^+$ if $n=2$.}
\end{cases}
\end{equation}
\end{lem}

\begin{proof}
  By \eqref{eq2.4} the formula for $u_0(r)$ makes sense and it is easy
  to check that $u=u_0(r)$ is a solution of \eqref{eq2.5} and, as
  $r\to 0^+$,
\[
 u_0(r) = \begin{cases}
           O(r^{2-n})&\text{if $n\ge 3$}\\
\noalign{\medskip}
O\left(\log \dfrac1r\right)&\text{if $n=2$.}
          \end{cases}
\]
Thus, since all solutions of \eqref{eq2.5} are given by
\[
 u= u_0(r) + C_1+C_2 
\begin{cases}
 r^{2-n}&\text{if $n\ge 3$}\\
\noalign{\medskip}
\log \dfrac1r&\text{if $n=2$}
\end{cases}
\]
where $C_1$ and $C_2$ are arbitrary constants, we see that all
solutions of \eqref{eq2.5} satisfy \eqref{eq2.6}.
\end{proof}

\begin{lem}\label{lem2.3}
  Suppose $f\colon (0,R]\to {\bb R}$ is a continuous function, $n\ge
  2$ is an integer, and
\begin{equation}\label{eq2.7}
\intl_{x\in B_R(0)\subset{\bb R}^n} |f(|x|)| \,dx < \infty.
\end{equation}
If $u = u(|x|)$ is a radial solution of
\begin{equation}\label{eq2.8}
 -\Delta^m u = f\quad \text{for}\quad 0<|x|\le R,\quad  m\ge 1
\end{equation}
then
\begin{equation}\label{eq2.9}
\intl_{|x|<r}  |u(x)|\,dx =
\begin{cases}
 O(r^2)&\text{as $r\to 0^+$ if $n\ge 3$}\\
\noalign{\medskip}
O\left(r^2 \log \dfrac1r\right)&\text{as $r\to 0^+$ if $n=2$.}
\end{cases}
\end{equation}
\end{lem}

\begin{proof}
  The lemma is true for $m=1$ by Lemma~\ref{lem2.2}. Assume,
  inductively, that the lemma is true for $m-1$ where $m\ge 2$. Let
  $u$ be a radial solution of \eqref{eq2.8}. Then
\[
 -\Delta(\Delta^{m-1}u) = -\Delta^mu = f\quad \text{for}\quad 0<|x|\le R.
\]
Hence by \eqref{eq2.7} and Lemma \ref{lem2.2},
\[
 g := -\Delta^{m-1}u \in L^1(B_R(0)).
\]
So by the inductive assumption, \eqref{eq2.9} holds.
\end{proof}

\begin{lem}\label{lem2.4}
  Suppose $f\colon \ovl{B_R(0)} - \{0\}\to {\bb R}$ is a nonnegative
  continuous function and $u$ is a $C^{2m}$ solution of
\begin{equation}\label{eq2.10}
 \left.\begin{array}{r} -\Delta^mu=f\\ u\ge 0\end{array}\right\} 
\quad \text{in} \quad \ovl{B_R(0)} - \{0\} \subset {\bb R}^n, 
\quad n\ge 2,\quad m\ge 1.
\end{equation}
Then
\begin{align}\label{eq2.11}
 &\intl_{|x|<r} u(x)\,dx = \begin{cases}
                          O(r^2)&\text{as $r\to 0^+$ if $n\ge 3$}\\
\noalign{\medskip}
O\left(r^2\log \dfrac1r\right)&\text{as $r\to 0^+$ if $n=2$}\end{cases}\\
\intertext{and}
\label{eq2.12}
&\intl_{|x|<R} |x|^{2m-2} f(x)\,dx < \infty.                         
\end{align}
\end{lem}

\begin{proof}
  By averaging \eqref{eq2.10} we can assume $f=f(|x|)$ and $u=u(|x|)$
  are radial functions. The lemma is true for $m=1$ by
  Lemmas~\ref{lem2.1} and \ref{lem2.2}. Assume inductively that the
  lemma is true for $m-1$, where $m\ge 2$. Let $u = u(|x|)$ be a
  radial solution of \eqref{eq2.10}. Let $v=\Delta^{m-1}u$. Then
  $-\Delta v = -\Delta^mu=f$ and integrating this equation we obtain
  as in the proof of Lemma~\ref{lem2.1} that
\begin{equation}\label{eq2.13}
 r^{n-1}v'(r) = r^{n-1}_2v'(r_2) + \int^{r_2}_r \rho^{n-1}f(\rho)\,d\rho
\quad \text{for all}\quad 0<r<r_2\le R.
\end{equation}

We can assume 
\begin{equation}\label{eq2.14}
 \int^R_0 \rho^{n-1}f(\rho)\,d\rho = \infty
\end{equation}
for otherwise $\intl_{|x|<R} f(x)\,dx < \infty$ and hence
\eqref{eq2.12} obviously holds and \eqref{eq2.11} holds by Lemma
\ref{lem2.3}. By \eqref{eq2.13} and \eqref{eq2.14} we have for some
$r_1\in (0,R)$ that
\begin{equation}\label{eq2.15}
 v'(r_1)\ge 1.
\end{equation}
Replacing $r_2$ with $r_1$ in \eqref{eq2.13} we get
\[
 v'(\rho) = \frac{r^{n-1}_1 v'(r_1)}{\rho^{n-1}} 
+ \frac1{\rho^{n-1}} \int^{r_1}_\rho s^{n-1}f(s)\,ds
\quad \text{for}\quad 0<\rho\le r_1
\]
and integrating this equation from $r$ to $r_1$ we obtain for $0<r\le
r_1$ that
\[
 -v(r) = -v(r_1) + r^{n-1}_1v'(r_1) \int^{r_1}_r \frac1{\rho^{n-1}}
 \,d\rho 
+ \int^{r_1}_r \frac1{\rho^{n-1}} \int^{r_1}_\rho s^{n-1} f(s)\,ds\,d\rho
\]
and hence by \eqref{eq2.15} for some $r_0\in (0,r_1)$ we have
\[
 -\Delta^{m-1}u(r) = -v(r) 
> \int^{r_0}_r \frac1{\rho^{n-1}} \int^{r_0}_\rho s^{n-1}f(s)\,ds \,d\rho 
\ge 0\quad \text{for}\quad 0<r\le r_0.
\]
So by the inductive assumption, $u$ satisfies \eqref{eq2.11} and 
\begin{align*}
\infty &> \frac1{n\omega_n} \intl_{|x|<r_0} |x|^{2m-4}(-v(|x|))\,dx\\
&= \int^{r_0}_0 r^{2m+n-5} (-v(r))\,dr\\
&\ge \int^{r_0}_0 r^{2m+n-5} \left(\int^{r_0}_r 
\frac1{\rho^{n-1}} \int^{r_0}_\rho s^{n-1} f(s)\,ds\,d\rho\right)dr\\
&= C\int^{r_0}_0 s^{2m-2} f(s)s^{n-1}ds\\
&= C\intl_{|x|<r_0} |x|^{2m-2} f(x)\,dx
\end{align*}
where in the above calculation we have interchanged the order of
integration and $C$ is a positive constant which depends only on $m$
and $n$. This completes the inductive proof.
\end{proof}

\section{Representation formula}\label{sec3}

\indent

A fundamental solution of $\Delta^m$ in ${\bb R}^n$, where $n\ge 2$
and $m\ge 1$ are integers, is given by
\begin{numcases}{\Phi(x):=a}
(-1)^m |x|^{2m-n}, & if $2 \le 2m < n$ \label{eq3.1}\\
(-1)^{\frac{n-1}2}|x|^{2m-n}, & if $3\le n < 2m$ and  $n$ is odd \label{eq3.2}\\
(-1)^{\frac{n}2} |x|^{2m-n} \log \frac5{|x|}, & if $2\le n \le 2m$ 
 and $n$ is even \label{eq3.3}
\end{numcases}
where $a = a(m,n)$ is a \emph{positive} constant. In the sense of
distributions, $\Delta^m\Phi=\delta$, where $\delta$ is the Dirac mass
at the origin in ${\bb R}^n$. For $x\ne 0$ and
$y\ne x$, let
\begin{equation}\label{eq3.4}
 \Psi(x,y) = \Phi(x-y) - \sum_{|\alpha|\le 2m-3} 
\frac{(-y)^\alpha}{\alpha!} D^\alpha\Phi(x)
\end{equation}
be the error in approximating $\Phi(x-y)$ with the partial sum of
degree $2m-3$ of the Taylor series of $\Phi$ at $x$.

The following theorem gives representation formula \eqref{eq3.6} for
nonnegative solutions of inequality \eqref{eq3.5}.

\begin{thm}\label{thm3.1}
Let $u(x)$ be a $C^{2m}$ nonnegative solution of
\begin{equation}\label{eq3.5}
 -\Delta^mu \ge 0\quad \text{in}\quad B_2(0)-\{0\}\subset {\bb R}^n,
\end{equation}
where $n\ge 2$ and $m\ge 1$ are integers. Then
\begin{equation}\label{eq3.6}
 u = N+h+ \sum_{|\alpha|\le 2m-2} a_\alpha D^\alpha\Phi\quad 
\text{ in}\quad B_1(0)-\{0\}
\end{equation}
where $a_\alpha, |\alpha|\le 2m-2$, are constants, $h\in
C^\infty(B_1(0))$ is a solution of
\[
 \Delta^mh = 0\quad \text{in}\quad B_1(0),
\]
and
\begin{equation}\label{eq3.7}
 N(x) = \intl_{|y|\le 1} \Psi(x,y) \Delta^mu(y)\,dy\quad \text{for}\quad x\ne 0.
\end{equation}
\end{thm}

When $m=1$, equation \eqref{eq3.6} becomes 
\[
u = N+h+a_0\Phi_1 \quad \text{in} \quad B_1(0)-\{0\},
\]
where
\[
 N(x)= \intl_{|y|<1} \Phi_1(x-y) \Delta u(y)\,dy
\]
and $\Phi_1$ is the fundamental solution of the Laplacian in ${\bb
  R}^n$. Thus, when $m=1$, Theorem~\ref{thm3.1} is essentially a
result of Brezis and Lions \cite{BL}.

Futamura, Kishi, and Mizuta \cite[Theorem 1]{FKM} and \cite[Corollary
5.1]{FM} obtained a result very similar to our Theorem~\ref{thm3.1},
but using their result we would have to let the index of summation
$\alpha$ in \eqref{eq3.4} range over the larger set $|\alpha|\le
2m-2$. This would not suffice for our proof of
Theorem~\ref{thm1.1}. We have however used their idea of using the
remainder term $\Psi(x,y)$ instead of $\Phi(x-y)$ in
\eqref{eq3.7}. This is done so that the integral in \eqref{eq3.7} is
finite. See also the book \cite[p. 137]{HK}.

\begin{proof}[Proof of Theorem \ref{thm3.1}]
By \eqref{eq3.5},
\begin{equation}\label{eq3.8}
f := -\Delta^mu \ge 0\quad \text{in}\quad B_2(0)-\{0\}.
\end{equation}
Thus by Lemma \ref{lem2.4},
\begin{equation}\label{eq3.9}
 \intl_{|x|<1} |x|^{2m-2} f(x)\,dx < \infty
\end{equation}
and
\begin{equation}\label{eq3.10}
 \intl_{|x|<r} u(x)\,dx = O\left(r^2\log \frac1r\right)
\quad \text{as}\quad r\to 0^+.
\end{equation}

If $|\alpha| = 2m-2$ we claim
\begin{equation}\label{eq3.11}
D^\alpha \Phi(x) = O(\Gamma_0(x))\quad \text{as}\quad x\to 0
\end{equation}
where $\Gamma_0(x)$ is given by \eqref{eq1.6}. This is clearly true if
$\Phi$ is given by \eqref{eq3.1} or \eqref{eq3.2} because then $n\ge
3$ and $\Gamma_0(x)=|x|^{2-n}$. The estimate \eqref{eq3.11} is also
true when $\Phi$ is given by \eqref{eq3.3} because then $|x|^{2m-n}$
is a \emph{polynomial} of degree $2m-n\le 2m-2=|\alpha|$ with equality
if and only if $n=2$, and hence $D^\alpha\Phi$ has a term with $\log
\frac5{|x|}$ as a factor if and only if $n=2$. This proves
\eqref{eq3.11}.

By Taylor's theorem and \eqref{eq3.11} we have
\begin{align}\label{eq3.12}
 |\Psi(x,y)| &\le C|y|^{2m-2} \Gamma_0(x)\\
&\le C|y|^{2m-2} |x|^{2-n} \log \frac5{|x|} \quad  
\text{for}\quad |y| < \frac{|x|}2<1.\notag
\end{align}

Differentiating \eqref{eq3.4} with respect to $x$ we get
\begin{equation}\label{eq3.13}
 D^\beta_x(\Psi(x,y)) = (D^\beta\Phi)(x-y) 
-  \sum_{|\alpha|\le 2m-3} \frac{(-y)^\alpha}{\alpha!} 
(D^{\alpha+\beta}\Phi)(x)\quad \text{for}\quad x\ne 0 
\quad \text{and}\quad y\ne x
\end{equation}
and so by Taylor's theorem applied to $D^\beta\Phi$ we have
\begin{equation}\label{eq3.14}
 |D^\beta_x\Psi(x,y)| \le C|y|^{2m-2} |x|^{2-n-|\beta|} 
\log \frac5{|x|}\quad \text{for}\quad |y| < \frac{|x|}2 <1.
\end{equation}
Also,
\begin{equation}\label{eq3.15}
 \Delta^m_x\Psi(x,y) = 0 = \Delta^m_y\Psi(x,y)\quad \text{for}
\quad x\ne 0\quad \text{and}\quad y\ne x
\end{equation}
(see also \cite[Lemma 4.1, p. 137]{HK}) and
\begin{align}
 \intl_{|x|<r} |\Phi(x-y)|\,dx &\le Cr^{2m}\log\frac{5}{r}\notag\\
&\le C|y|^{2m-2} r^2\log \frac5r 
\quad \text{for} \quad 0 < r \le 2|y| < 2.
\label{eq3.16}
\end{align}
Before continuing with the proof of Theorem \ref{thm3.1}, we state and
prove the following lemma.

\begin{lem}\label{lem3.1}
For $|y|<1$ and $0<r<1$ we have
\begin{equation}\label{eq3.17}
 \intl_{|x|<r} |\Psi(x,y)| \,dx \le C|y|^{2m-2} r^2\log \frac5r.
\end{equation}
\end{lem}

\begin{proof}
Since $\Psi(x,0)\equiv 0$ for $x\ne 0$, we can assume $y\ne 0$.\medskip

\n {\bf Case I.} Suppose $0<r\le |y| <1$. Then by \eqref{eq3.16}
\begin{align*}
\intl_{0<|x|<r} |\Psi(x,y)|\,dx &\le \intl_{0<|x|<r} |\Phi(x-y)|\,dx 
+ \sum_{|\alpha|\le 2m-3} |y|^{|\alpha|} \intl_{0<|x|<r} |D^\alpha\Phi(x)|\,dx\\
&\le C\left[|y|^{2m-2} r^2 \log \frac5r 
+ \sum_{|\alpha|<2m-3} |y|^{|\alpha|} r^{2m-|\alpha|} \log \frac5r\right]\\
&\le C |y|^{2m-2} r^2 \log \frac5r.
\end{align*}

\n {\bf Case II.} Suppose $0<|y|<r<1$. Then by \eqref{eq3.16}, with
$r=2|y|$, and \eqref{eq3.12} we have
\begin{align*}
\intl_{|x|<2r} |\Psi(x,y)|\,dx &= \intl_{2|y|<|x|<2r} |\Psi(x,y)|\,dx 
+ \intl_{|x|<2|y|} |\Psi(x,y)|\,dx\\
&\le C\left[\, \intl_{2|y|<|x|<2r} |y|^{2m-2} |x|^{2-n} 
\log \frac5{|x|} \,dx + |y|^{2m} \log \frac5{|y|}\right.\\
&\quad \left. + \sum_{|\alpha|\le 2m-3} |y|^{|\alpha|} \intl_{|x|<2|y|} 
|D^\alpha\Phi(x)|\,dx\right]\\
&\le C\left[|y|^{2m-2} r^2 \log \frac5r 
+ |y|^{2m-2} |y|^2 \log \frac5{|y|}\right]\\
&\le C |y|^{2m-2} r^2 \log \frac5r
\end{align*}
which proves the lemma.
\end{proof}

Continuing with the proof of Theorem \ref{thm3.1},
let $N$ be defined by \eqref{eq3.7} and let $2r\in (0,1)$ be
fixed. Then for $2r<|x|<1$ we have
\begin{align*}
 N(x) &= \intl_{r<|y|<1} \left[\Phi (y-x) - \sum_{|\alpha|\le 2m-3} 
\frac{(-y)^\alpha}{\alpha!} D^\alpha\Phi(x)\right] \Delta^m u(y)\,dy\\
&\quad -\intl_{0<|y|<r} \Psi(x,y) f(y)\,dy.
\end{align*}
By \eqref{eq3.9} and \eqref{eq3.14}, we can move
differentiation of the second integral with respect to $x$ under the
integral. Hence by \eqref{eq3.15},
\begin{equation}\label{eq3.18}
 \Delta^m N = \Delta^m u
\end{equation}
for $2r< |x|<1$ and since $2r\in (0,1)$ was arbitrary, \eqref{eq3.18}
holds for $0<|x|<1$.

By \eqref{eq3.7}, \eqref{eq3.8}, and Lemma \ref{lem3.1}, for $0<r<1$ we have
\begin{align*}
 \intl_{|x|<r} |N(x)|\,dx &\le \intl_{|y|<1} 
\left(\, \intl_{|x|<r} |\Psi(x,y)|\,dx\right) f(y)\,dy\\
&\le Cr^2 \log \frac5r \intl_{|y|<1} |y|^{2m-2} f(y)\,dy\\
&= O\left(r^2 \log \frac1r\right) \quad \text{as}\quad r\to 0^+
\end{align*}
by \eqref{eq3.9}. Thus by \eqref{eq3.10}
\begin{equation}\label{eq3.19}
 v := u-N \in L^1_{\text{loc}}(B_1(0)) \subset {\cl D}'(B_1(0))
\end{equation}
and
\begin{equation}\label{eq3.20}
 \intl_{|x|<r} |v(x)|\,dx = O\left(r^2\log \frac1r\right) \quad 
\text{as}\quad r\to 0^+.
\end{equation}
By \eqref{eq3.18},
\[
 \Delta^mv(x) = 0\quad \text{for}\quad 0<|x|<1.
\]
Thus $\Delta^mv$ is a distribution in ${\cl D}'(B_1(0))$ whose support
is a subset of $\{0\}$. Hence
\[
 \Delta^mv = \sum_{|\alpha|\le k} a_\alpha D^\alpha\delta
\]
is a finite linear combination of the delta function and its derivatives.

We now use a method of Brezis and Lions \cite{BL} to show $a_\alpha=0$ for
$|\alpha|\ge 2m-1$. Choose $\varphi\in C^\infty_0(B_1(0))$ such that
\[
(-1)^{|\alpha|}(D^\alpha\varphi)(0)=a_\alpha\quad \text{for}\quad  |\alpha|\le k.
\]
Let $\varphi_\vp(x) = \varphi\big(\frac{x}\vp\big)$. Then, for
$0<\vp<1$, $\varphi_\vp\in C^\infty_0(B_1(0))$ and
\begin{align*}
 \int v\Delta^m\varphi_\vp &= (\Delta^mv)(\varphi_\vp) 
= \sum_{|\alpha|\le k} a_\alpha(D^\alpha\delta)\varphi_\vp\\
&= \sum_{|\alpha|\le k} a_\alpha(-1)^{|\alpha|}
\delta(D^\alpha\varphi_\vp) 
= \sum_{|\alpha|\le k} a_\alpha(-1)^{|\alpha|} (D^\alpha\varphi_\vp)(0)\\
&= \sum_{|\alpha|\le k} a_\alpha(-1)^{|\alpha|} \frac1{\vp^{|\alpha|}} 
(D^\alpha\varphi)(0) = \sum_{|\alpha|\le k} a^2_\alpha \frac1{\vp^{|\alpha|}}.
\end{align*}

On the other hand,
\begin{align*}
 \int v\Delta^m \varphi_\vp &= \int v(x) \frac1{\vp^{2m}} 
(\Delta^m\varphi) \left(\frac{x}\vp\right) \,dx\\
&\le \frac{C}{\vp^{2m}} \intl_{|x|<\vp} |v(x)|\,dx 
= O\left(\frac1{\vp^{2m-2}} \log \frac1\vp\right) \quad 
\text{as}\quad \vp\to 0^+
\end{align*}
by \eqref{eq3.20}. Hence $a_\alpha = 0$ for $|\alpha|\ge 2m-1$ and consequently
\[
 \Delta^mv = \sum_{|\alpha|\le 2m-2} a_\alpha D^\alpha\delta 
= \sum_{|\alpha|\le 2m-2} a_\alpha D^\alpha \Delta^m\Phi.
\]
That is
\[
 \Delta^m \left(v-\sum_{|\alpha|\le 2m-2} a_\alpha D^\alpha\Phi\right) 
= 0 \quad \text{in}\quad {\cl D}'(B_1(0)).
\]
Thus for some $C^\infty$ solution of $\Delta^mh = 0$ in $B_1(0)$ we have
\[
 v = \sum_{|\alpha|\le 2m-2} a_\alpha D^\alpha \Phi 
+ h\quad \text{in}\quad B_1(0) - \{0\}.
\]
Hence Theorem \ref{thm3.1} follows from \eqref{eq3.19}.
\end{proof}

\section{Proofs of Theorems \ref{thm1.3} and \ref{thm1.4} and
Corollary \ref{cor1.1}}
\label{sec4}

\indent

In this section we prove Theorems \ref{thm1.3} and \ref{thm1.4} and
Corollary \ref{cor1.1}.

\begin{proof}[Proof of Theorem \ref{thm1.3}]
This proof is a continuation of the proof of Theorem~\ref{thm3.1}.
If $m=1$ then Theorem~\ref{thm1.3} is trivially true. Hence we can
assume $m\ge 2$. Also, if $\sigma=m$ then \eqref{eq4.2} follows
trivially from \eqref{eq4.1}. Hence we can assume $\sigma\le m-1$ in
\eqref{eq4.2}.

If $\alpha$ and $\beta$ are multi-indices and $|\alpha|= 2m-2$ then it
follows from \eqref{eq3.1}--\eqref{eq3.3} that
\begin{equation}\label{eq4.4}
 D^{\alpha+\beta}\Phi(x) =
 O\left(\left|\frac{d^{|\beta|}}{d|x|^{|\beta|}} 
\Gamma_0(|x|)\right|\right) \quad \text{as}\quad x\to 0.
\end{equation}
(This is clearly true if $n=2$. If $n\ge 3$ then $|\alpha+\beta| =
2m-2+|\beta| > 2m-n$ and thus
\[
D^{\alpha+\beta} \Phi(x) = O(|x|^{2m-n-(2m-2+|\beta|)}) 
= O\left(\left|\frac{d^{|\beta|}}{d|x|^{|\beta|}} 
\Gamma_0(|x|)\right|\right).)
\]
Let $L^b$ be any linear partial differential operator of the form
$\sum\limits_{|\beta|=b} c_\beta D^\beta$, where $b$ is a nonnegative
integer and $c_\beta\in {\bb R}$. Then applying Taylor's theorem to
\eqref{eq3.13} and using \eqref{eq4.4} we obtain
\begin{equation}\label{eq4.5}
 |L^b_x\Psi(x,y)|\le C|y|^{2m-2} \left|\frac{d^b}{d|x|^b} 
\Gamma_0(|x|)\right| \quad \text{for}\quad |y|< \frac{|x|}2<1.
\end{equation}
Here and later $C$ is a positive constant, independent of $x$ and $y$,
whose value may change from line to line. For $0\le b\le 2m-1$ we have
\[
 L^bN(x)=\int\limits_{|y|<1} -L^b_x \Psi(x,y) f(y)\,dy\quad 
\text{for}\quad 0<|x|<1.
\]
Hence by \eqref{eq4.4}, \eqref{eq4.5}, \eqref{eq3.6} and \eqref{eq3.9} we have
\begin{equation}\label{eq4.6}
 L^bu(x) \le C\left|\frac{d^b}{d|x|^b} 
\Gamma_0(|x|)\right|\quad \text{for}\quad 0<|x|<1
\end{equation}
provided $0\le b\le 2m-1$ and
\begin{equation}\label{eq4.7}
 -L^b_x\Psi(x,y) \le C|y|^{2m-2} \left|\frac{d^b}{d|x|^b} 
\Gamma_0(|x|)\right|\quad \text{for}\quad 0<\frac{|x|}2 < |y|<1.
\end{equation}
We will complete the proof of Theorem~\ref{thm1.3} by proving
\eqref{eq4.7} for various choices for $L^b$. For the rest of the proof
of Theorem~\ref{thm1.3} we will always assume
\begin{equation}\label{eq4.8}
 0 < \frac{|x|}2 < |y|<1
\end{equation}
which implies
\begin{equation}\label{eq4.9}
 |x-y| \le |x|+|y| \le 3|y|.
\end{equation}

\n {\bf Case I.} Suppose $\Phi$ is given by \eqref{eq3.1} or
\eqref{eq3.2}. It follows from \eqref{eq3.13} and \eqref{eq4.8} that
\begin{align*}
 |D^\beta_x\Psi(x,y) - D^\beta_x\Phi(x-y)| &\le 
C\sum_{|\alpha|\le 2m-3} 
|y|^{|\alpha|} |x|^{2m-n-|\alpha|-|\beta|}\\
&\le C|y|^{2m-2} |x|^{2-n-|\beta|}.
\end{align*}
Thus \eqref{eq4.7}, and hence \eqref{eq4.6}, holds provided $0\le b\le
2m-1$ and
\begin{equation}\label{eq4.10}
-(L^b\Phi)(x-y) \le C|y|^{2m-2} |x|^{2-n-b}.
\end{equation}

\n {\bf Case I(a).} Suppose $\Phi$ is given by \eqref{eq3.1}. Let
$\sigma\in [0,m-1]$ be an integer, $b=2\sigma$, and
$L^b=(-1)^{m+\sigma}\Delta^\sigma$. Then $0 \le b\le 2m-2$ and
\[
\text{sgn}(-L^b\Phi) = (-1)^{1+m+\sigma} \text{ sgn } 
\Delta^\sigma\Phi = (-1)^{1+2m+\sigma} \text{ sgn } 
\Delta^\sigma|x|^{2m-n} = (-1)^{1+2m+2\sigma}=-1.
\]
Thus \eqref{eq4.10}, and hence \eqref{eq4.6} holds with
$L^b=(-1)^{m+\sigma}\Delta^\sigma$ and $0\le \sigma \le m-1$. This
completes the proof of Theorem~\ref{thm1.3} when $\Phi$ is given
by \eqref{eq3.1}.\medskip

\n {\bf Case I(b).} Suppose $\Phi$ is given by \eqref{eq3.2}. Then $n$
is odd. It follows from \eqref{eq4.8} and \eqref{eq4.9} that for $0\le
|\beta|\le 2m-n$ we have
\[
 |(D^\beta\Phi)(x-y)| \le C|x-y|^{2m-n-|\beta|} \le
 C|y|^{2m-n-|\beta|} 
\le C|y|^{2m-2} |x|^{2-n-|\beta|}.
\]
So \eqref{eq4.10} holds with $L^b = \pm D^\beta$ and $|\beta|=b$. Hence
\[
 |D^\beta u(x)| \le C|x|^{2-n-|\beta|} \quad \text{for}\quad 0 
\le |\beta|\le 2m-n\quad \text{and} \quad 0<|x|<1.
\]
In particular
\[
 |\Delta^\sigma u(x)|\le C|x|^{2-n-2\sigma} \quad \text{for}
\quad 2\sigma \le 2m-n\quad \text{and}\quad 0<|x|<1.
\]
Also, if $2m-n+1 \le 2\sigma \le 2m-2$, $b=2\sigma$, and $L^b =
(-1)^{m+\sigma}\Delta^\sigma$, then $0 \le \sigma\le m-1$ and
\begin{align*}
 \text{sgn} (-L^b\Phi) &= (-1)^{m+\sigma+1} \text{ sgn } 
\Delta^\sigma\Phi = (-1)^{m+\sigma+1+\frac{n-1}2} \text{ sgn } 
\Delta^\sigma |x|^{2m-n}\\
&= (-1)^{m+\sigma+1+\frac{n-1}2} \text{ sgn}
(\Delta^{\frac{b-(2m-n+1)}2} 
\Delta^{\frac{2m-n+1}2} |x|^{2m-n})\\ 
&= (-1)^{m+\sigma+1+\frac{n-1}2 +\sigma-m+\frac{n-1}2} = -1
\end{align*}
because $\Delta^{\frac{2m-n+1}2} |x|^{2m-n} = C|x|^{-1}$ where $C>0$.

So \eqref{eq4.10} holds with $L^b =
(-1)^{m+\sigma}\Delta^\sigma$. Hence $(-1)^{m+\sigma} \Delta^\sigma
u(x) \le C|x|^{2-n-2\sigma}$ for $0 \le \sigma\le m-1$ and
$0<|x|<1$. This completes the proof Theorem~\ref{thm1.3} when $\Phi$
is given by \eqref{eq3.2}.\medskip

\n {\bf Case II.} Suppose $\Phi$ is given by \eqref{eq3.3}. Then $2\le
n \le 2m$ and $n$ is even. To prove Theorem~\ref{thm1.3} in Case~II,
it suffices to prove the following three statements.
\begin{itemize}
 \item[(i)] Estimate \eqref{eq4.3} holds when $n=2$, $\beta=0$, and $m\ge 2$.
\item[(ii)] Estimate \eqref{eq4.3} holds when $|\beta|\le 2m-n-1$ 
            and either $n\ge 3$ or $|\beta|\ge 1$.
\item[(iii)] Estimate \eqref{eq4.2} holds for $2m-n\le 2\sigma\le 2m-2$.
\end{itemize}
\medskip

\noindent {\it Proof of (i).}
Suppose $n=2$, $\beta=0$, and $m \ge 2$. Then, since $u$ is 
nonnegative, to prove (i) it suffices to prove
\[
 u(x) \le C \log \frac5{|x|}\quad \text{for}\quad 0<|x|<1
\]
which holds if \eqref{eq4.7} holds with $b=0$ and $L^b=D^0 =$ id. That
is if
\begin{equation}\label{eq4.11}
-\Psi(x,y) \le C|y|^{2m-2} \log \frac5{|x|}
\end{equation}
By \eqref{eq3.4}, \eqref{eq4.8}, and \eqref{eq4.9} we have
\begin{align*}
 |\Psi(x,y) - \Phi(x-y)| &\le \sum_{|\alpha|\le 2m-3} |y|^{|\alpha|} 
|D^\alpha\Phi(x)|\\
&\le C \sum_{|\alpha|\le 2m-3} |y|^{|\alpha|} |x|^{2m-2-|\alpha|} 
\log \frac5{|x|} \le C |y|^{2m-2} \log \frac5{|x|}
\end{align*}
and
\begin{align*}
 |\Phi(x-y)| &= a|x-y|^{2m-2} \log \frac5{|x-y|}\\
&\le C|y|^{2m-2} \log \frac5{|y|} \le C|y|^{2m-2} \log \frac5{|x|}
\end{align*}
which imply \eqref{eq4.11}. This completes the proof of (i).
\medskip

\noindent {\it Proof of (ii).}
Suppose $|\beta|\le 2m-n-1$ and either $n\ge 3$ or $|\beta|\ge
  1$. Then $n+|\beta|\ge 3$ and in
order to prove (ii) it suffices to prove
\begin{equation}\label{eq4.12}
 |D^\beta_x\Psi(x,y)|\le C|y|^{2m-2} 
\left|\frac{d^{|\beta|}}{d|x|^{|\beta|}} \Gamma_0(|x|)\right|
\end{equation}
because then \eqref{eq4.7}, and hence \eqref{eq4.6}, holds with
$L^b=\pm D^\beta$.

Since $\Phi$ is given by \eqref{eq3.3} we have $n\ge 2$ is even and
\[
 \Phi(x) = P(x) \log \frac5{|x|}
\]
where $P(x) = a(-1)^{\frac{n}2} |x|^{2m-n}$ is a \emph{polynomial} of
degree $2m-n$. Since $D^\beta P$ is a polynomial of degree
$2m-n-|\beta|\le 2m-3$ we have
\begin{equation}\label{eq4.13}
 D^\beta_xP(x-y) = \sum_{|\alpha|\le 2m-3} 
\frac{(-y)^\alpha}{\alpha!} D^{\alpha+\beta}P(x).
\end{equation}

Since $D^\beta_x\Psi(x,y) = A_1+A_2+A_3$, where
\begin{align*}
 A_1 &= D^\beta_x\Psi(x,y) - D^\beta_x\Phi(x-y) + (D^\beta_xP(x-y)) 
\log \frac5{|x|}\\
A_2 &= D^\beta_x\Phi(x-y) - (D^\beta_xP(x-y)) \log \frac5{|x-y|}\\
A_3 &= (D^\beta_xP(x-y)) \log \frac{|x|}{|x-y|},
\end{align*}
to prove \eqref{eq4.12} it suffices to prove for $j=1,2,3$ that
\begin{equation}\label{eq4.14}
 |A_j| \le C|y|^{2m-2} \left|\frac{d^{|\beta|}}{d|x|^{|\beta|}} 
\Gamma_0(|x|)\right|.
\end{equation}

Since
\begin{align*}
 \left|D^{\alpha+\beta}\Phi(x) - (D^{\alpha+\beta}P(x)) 
\log \frac5{|x|}\right| &= \left|
\sum_{\underset{\scriptstyle |\alpha+\beta-\gamma|\ge 1}{\gamma\le
    \alpha+\beta}} 
\binom{\alpha+\beta}\gamma (D^\gamma P(x))
\left(D^{\alpha+\beta-\gamma} \log \frac5{|x|}\right)\right|\\
&\le C |x|^{2m-n-|\alpha|-|\beta|}
\end{align*}
it follows from \eqref{eq3.13}, \eqref{eq4.13}, and \eqref{eq4.8} that
\begin{align*}
 |A_1| = |-A_1| &= \left|\sum_{|\alpha|\le 2m-3} 
\frac{(-y)^\alpha}{\alpha!} D^{\alpha+\beta} \Phi(x) 
- \sum_{|\alpha|\le 2m-3} \frac{(-y)^\alpha}{\alpha!} 
(D^{\alpha+\beta} P(x)) \log \frac5{|x|}\right|\\
&\le C \sum_{|\alpha|\le 2m-3} |y|^{|\alpha|}
|x|^{2m-n-|\alpha|-|\beta|} 
\le C |y|^{2m-2} |x|^{2-n-|\beta|}\\
&= C|y|^{2m-2} \left|\frac{d^{|\beta|}}{d|x|^{|\beta|}} \Gamma_0(|x|)\right|.
\end{align*}
Thus \eqref{eq4.14} hold when $j=1$.

Since $A_2=0$ when $\beta=0$, we can assume for the proof of
\eqref{eq4.14} when $j=2$ that $|\beta|\ge 1$. Then by \eqref{eq4.9}
and \eqref{eq4.8},
\begin{align*}
 |A_2| &= \left|\sum_{\underset{\scriptstyle |\beta-\alpha|\ge 1}{\alpha\le \beta}} 
\binom\beta\alpha (D^\alpha_xP(x-y)) \left(D^{\beta-\alpha}_x 
\log \frac5{|x-y|}\right)\right|\\
&\le C|x-y|^{2m-n-|\beta|} \le C|y|^{2m-n-|\beta|}\\
&\le C |y|^{2m-2} |x|^{2-n-|\beta|}\\
&= C|y|^{2m-2} \left|\frac{d^{|\beta|}}{d|x|^{|\beta|}} \Gamma_0(|x|)\right|.
\end{align*}
Thus \eqref{eq4.14} holds when $j=2$.

Finally we prove \eqref{eq4.14} when $j=3$. Let $d=2m-n-|\beta|$. Then
$1\le d\le 2m-3$,
\[
 |A_3| \le C|x-y|^d \left|\log \frac{|x|}{|x-y|}\right|
\]
and by \eqref{eq4.8} and \eqref{eq4.9} we have
\begin{align*}
 |x-y|^d \left|\log \frac{|x|}{|x-y|}\right| &\le
\begin{cases}
 |x-y|^d \left(\dfrac{|x|}{|x-y|}\right)^d 
= |x|^d \le C|y|^{2m-2} |x|^{2-n-|\beta|} &\text{if $|x-y|\le |x|$}\\
|x-y|^d \left(\dfrac{|x-y|}{|x|}\right)^{2m-2-d} 
= |x-y|^{2m-2} |x|^{2-n-|\beta|} &\text{if $|x|\le |x-y|$}
\end{cases}\\
&\le C|y|^{2m-2} |x|^{2-n-|\beta|} 
= C|y|^{2m-2} \left|\frac{d^{|\beta|}}{d|x|^{|\beta|}} \Gamma_0(|x|)\right|.
\end{align*}
Thus \eqref{eq4.14} holds when $j=3$. This completes the proof of
\eqref{eq4.12} and hence of (ii).
\medskip

\noindent {\it Proof of (iii).}
Suppose $2m-n\le 2\sigma \le 2m-2$. In order to prove (iii) it suffices to prove
\begin{equation}\label{eq4.15}
 (-1)^{m+\sigma+1} \Delta^\sigma_x\Psi(x,y) \le C|y|^{2m-2} 
\left|\frac{d^{2\sigma}}{d|x|^{2\sigma}} \Gamma_0(|x|)\right|
\end{equation}
because then \eqref{eq4.7}, and hence \eqref{eq4.6}, holds with
$L^b=(-1)^{m+\sigma}\Delta^\sigma$ and $b=2\sigma$.

If $|\beta|=2\sigma$ then \eqref{eq4.8} implies
\begin{align*}
\left|\sum_{1\le |\alpha| \le 2m-3} \frac{(-y)^\alpha}{\alpha!} 
D^{\alpha+\beta} \Phi(x)\right| &\le C \sum_{1\le |\alpha|\le 2m-3} 
|y|^{|\alpha|} |x|^{2m-n-|\alpha|-|\beta|}\\
&\le C|y|^{2m-2} |x|^{2-n-|\beta|}.
\end{align*}
Thus it follows from \eqref{eq3.13} that
\[
 |\Delta^\sigma_x\Psi(x,y) - \Delta^\sigma_x\Phi(x-y) 
+ \Delta^\sigma\Phi(x)|\le C|y|^{2m-2} |x|^{2-n-2\sigma}.
\]
Hence to prove \eqref{eq4.15} it suffices to prove
\begin{equation}\label{eq4.16}
 (-1)^{m+\sigma+1} (\Delta^\sigma_x \Phi(x-y) - \Delta^\sigma\Phi(x)) 
\le C|y|^{2m-2} |x|^{2-n-2\sigma}.
\end{equation}
We divide the proof of \eqref{eq4.16} into cases.\medskip

\n {\bf Case 1.} Suppose $2\le 2m-n+2\le 2\sigma\le 2m-2$. Then by \eqref{eq4.8}
\[
 |\Delta^\sigma\Phi(x)| \le C|x|^{2m-n-2\sigma} \le C|y|^{2m-2} |x|^{2-n-2\sigma}
\]
and since
\begin{equation}\label{eq4.17}
 \Delta^{\frac{2m-n}2} \left(|x|^{2m-n} \log \frac5{|x|}\right) 
= A \log \frac5{|x|}-B
\end{equation}
where $A>0$ and $B\ge 0$ are constants, we have
\[
 \text{sgn}((-1)^{m+\sigma+1} \Delta^\sigma\Phi(z)) 
= (-1)^{m+\sigma+\frac{n}2+1} (-1)^{\sigma - \frac{2m-n}2} 
= -1 \quad \text{for}\quad |z|>0.
\]
This proves \eqref{eq4.16} and hence (iii) in Case 1.
\medskip 

\n {\bf Case 2.} Suppose $2\sigma =2m-n$. Then by \eqref{eq4.17} and
\eqref{eq4.9} we have
\begin{align*}
 (-1)^{m+\sigma+1} (\Delta^\sigma_x\Phi(x-y) - \Delta^\sigma\Phi(x)) 
&= (-1)^{\frac{n}2 +m+\sigma+1} A \log \frac{|x|}{|x-y|}\\
&= A\log \frac{|x-y|}{|x|} \le A \log \frac{3|y|}{|x|} 
\le A \left(\frac{3|y|}{|x|}\right)^{2m-2}\\
&= A 3^{2m-2} |y|^{2m-2} |x|^{2-n-2\sigma}.
\end{align*}
This proves \eqref{eq4.16} and hence (iii) in Case 2, and thereby
completes the proof of Theorem~\ref{thm1.3}.
\end{proof}

\begin{proof}[Proof of Theorem \ref{thm1.4}]
Let $u(x)$ be defined in terms of $v(y)$ by \eqref{eq1.5.1}. Then by
\eqref{eq1.5.2} and \eqref{neq4.1}, $u(x)$ is a $C^{2m}$ nonnegative
solution of \eqref{eq4.1}, and hence $u(x)$ satisfies the conclusion
of Theorem \ref{thm1.3}. It is a straight-forward exercise to show
that \eqref{neq4.3} follows from \eqref{eq4.3} when $n<2m$ and $\beta$
satisfies \eqref{neq4.4}. So to complete the proof of Theorem
\ref{thm1.4} we will now prove \eqref{neq4.2}.

Suppose $\sigma\le m$ is a nonnegative integer. Let $v_\sigma(y)$ be
the $\sigma$-Kelvin transform of $u(x)$. Then  
$v_\sigma(y)=|y|^{2\sigma-2m}v(y)$ and thus by \eqref{eq4.2}, we have
for $|y|>1$ that
\begin{align*}
(-1)^{m+\sigma}\Delta^\sigma(|y|^{2\sigma-2m}v(y))
&=(-1)^{m+\sigma}\Delta^\sigma v_\sigma(y)\\
&=(-1)^{m+\sigma}|x|^{n+2\sigma}\Delta^\sigma u(x)\\
&\le C|x|^{n+2\sigma} \left|\frac{d^{2\sigma}}{d|x|^{2\sigma}} 
\Gamma_0(|x|)\right|\\
&\le C\begin{cases}
|x|^2\log\frac{5}{|x|} 
&\text{if $\sigma=0$ and $n=2$}\\
|x|^2
&\text{if $\sigma\ge 1$ or $n\ge 3$}
\end{cases}
\end{align*}
which implies \eqref{neq4.2} after replacing $|x|$ with $1/|y|$.
\end{proof}

\begin{proof}[Proof of Corollary \ref{cor1.1}]
Theorem \ref{thm1.4} implies \eqref{eq1.14.1} and 
\[
-\Delta(|y|^{-2}v(y))\le C|y|^{-2} \qquad \text{for} \quad |y|>1 
\]
and thus for $|y|>1$ we have
\begin{align*}
-|y|^{-2}\Delta v(y)&=-\Delta(|y|^{-2} v(y))+(\Delta|y|^{-2})v(y) 
+ 2\nabla|y|^{-2}\cdot\nabla v(y)\\
&\le -\Delta(|y|^{-2} v(y))+C\left(|y|^{-4}\Gamma_\infty(|y|)+|y|^{-3}
\frac{d}{d|y|} \Gamma_\infty(|y|)\right)\\
&\le C\begin{cases}
|y|^{-2} & \text{if $n=3$}\\
|y|^{-2}\log 5|y| & \text{if $n=2$}
\end{cases}\\
&\le C|y|^{-2}\left|\frac{d^2}{d|y|^2} \Gamma_\infty(|y|)\right|
\end{align*}
which implies \eqref{eq1.14.2}.
\end{proof}

\section{Proof of Theorem \ref{thm1.1}}
\label{sec5}

\indent

As noted in the introduction, 
the sufficiency of condition \eqref{eq1.3} in Theorem \ref{thm1.1} and
the estimate \eqref{eq1.5} follow from Theorem \ref{thm1.3}, which we
proved in the last section. Consequently, we can complete the proof of
Theorem \ref{thm1.1} by proving the following proposition.

\begin{pro}\label{pro5.1}
  Suppose $n\ge 2$ and $m\ge 1$ are integers such that \eqref{eq1.3}
  does not hold. Let $\psi\colon (0,1)\to (0,\infty)$ be a
  continuous function. Then there exists a $C^\infty$ positive solution
  of
\begin{equation}\label{eq4.3.1}
 -\Delta^m u \ge 0\quad \text{in}\quad B_1(0)-\{0\}\subset {\bb R}^n
\end{equation}
such that
\begin{equation}\label{eq4.3.2}
 u(x)\ne O(\psi(|x|)) \quad \text{as}\quad x\to 0.
\end{equation}
\end{pro}

\begin{proof}
  Let $\{x_j\}^\infty_{j=1} \subset {\bb R}^n-\{0\}$ be a sequence
  such that $4|x_{j+1}| < |x_j| < 1$. Choose $\alpha_j>0$ such that
\begin{equation}\label{eq4.18}
 \frac{\alpha_j}{\psi(x_j)}\to \infty \quad \text{as}\quad j\to \infty.
\end{equation}
Since \eqref{eq1.3} does not hold, it follows from
\eqref{eq3.1}--\eqref{eq3.3} that $\lim\limits_{x\to 0} - \Phi(x) =
\infty$ and $-\Phi(x) > 0$ for $0<|x|<5$. Hence we can choose $R_j\in
(0, |x_j|/4)$ such that
\begin{equation}\label{eq4.19}
\intl_{|z|<R_j} - \Phi(z)\,dz > R^n_j 2^j\alpha_j,\quad \text{for}
\quad j=1,2,\ldots~.
\end{equation}
Let $\varphi\colon {\bb R}\to [0,1]$ be a $C^\infty$ function such
that $\varphi(t) = 1$ for $t\le 1$ and $\varphi(t) = 0$ for $t\ge
2$. Define $f_j\in C^\infty_0(B_{\frac{|x_j|}2}(x_j))$ by
\[
 f_j(x) = \frac1{2^jR^n_j} \varphi\left(\frac{|x-x_j|}{R_j}\right).
\]
Then the functions $f_j$ have disjoint supports and
\[
 \intl_{{\bb R}^n} f_j(x)\,dx = \intl_{|x-x_j|<2R_j} f_j(x)\,dx 
\le \frac{C(n)}{2^j}.
\]
Thus $f := \sum\limits^\infty_{j=1} f_j\in L^1({\bb R}^n) \cap
C^\infty({\bb R}^n - \{0\})$ and hence the function $u\colon B_1(0)
- \{0\}\to {\bb R}$ defined by
\[
 u(x) := \intl_{|y|<1} - \Phi(x-y) f(y)\,dy
\]
is a $C^\infty$ positive solution of \eqref{eq4.3.1}. Also
\begin{align*}
u(x_j) &\ge \intl_{|y|<1} - \Phi(x_j-y) f_j(y)\,dy\\
&\ge \frac1{2^jR^n_j} \intl_{|x-x_j|<R_j} - \Phi(x_j-y)\,dy\\
&= \frac1{2^jR^n_j} \intl_{|z|<R_j} - \Phi(z)\,dz > \alpha_j
\end{align*}
by \eqref{eq4.19}. Hence \eqref{eq4.18} implies that $u$ satisfies
\eqref{eq4.3.2}.
\end{proof}


\begin{thebibliography}{10}

\bibitem{ACP} N. Aronszajn, T.M. Creese, and L.J. Lipkin, Polyharmonic
  Functions, Oxford University Press, New York, 1983.

\bibitem{BL} H. Brezis and P.-L. Lions, A note on isolated
  singularities for linear elliptic equations,  Mathematical analysis
  and applications, Part A, pp. 263--266, Adv. in Math. Suppl. Stud.,
  7a, Academic Press, New York-London, 1981.

\bibitem{CDM} G. Caristi, L. D'Ambrosio, and E. Mitidieri,
  Representation formulae for solutions to some classes of higher
  order systems and related Liouville theorems, Milan J. Math. {\bf
    76} (2008), 27--67.

\bibitem{CMS} G. Caristi, E. Mitidieri, and R. Soranzo, Isolated
  singularities of polyharmonic equations, Atti
  Sem. Mat. Fis. Univ. Modena {\bf 46} (1998), 257--294.

\bibitem{CX} Y.S. Choi and X. Xu, Nonlinear biharmonic equations with
  negative exponents, J. Differential Equations {\bf 246} (2009),
  216--234.

\bibitem{FKM} T. Futamura, K. Kishi, and Y. Mizuta, Removability of
  sets for sub-polyharmonic functions, Hiroshima Math. J.  {\bf 33}
  (2003), 31--42.

\bibitem{FM} T. Futamura and Y. Mizuta, Isolated singularities of 
  super-polyharmonic functions, Hokkaido Math. J.  {\bf 33}
  (2004), 675--695.

\bibitem{GGS} F. Gazzola, H.-C. Grunau, and G. Sweers, Polyharmonic
  Boundary Value Problems, Springer, 2010.

\bibitem{GL} Y. Guo and J. Liu, Liouville-type theorems for
  polyharmonic equations in ${\bb R}^N$ and in ${\bb R}_+^N$,
  Proc. Roy. Soc. Edinburgh Sect. A {\bf 138} (2008), 339--359.

\bibitem{HK} W.K. Hayman and P.B. Kennedy, Subharmonic functions,
  Vol. I, Academic Press, London-New York, 1976.

\bibitem{H} S.-Y. Hsu, Removable singularity of the polyharmonic
  equation, Nonlinear Anal. {\bf 72} (2010), 624--627.

\bibitem{MR} P.J. McKenna and W. Reichel, Radial solutions of singular
  nonlinear biharmonic equations and applications to conformal
  geometry,  Electron. J. Differential Equations 2003, No. 37, 13~pp.

\bibitem{RW1} W. Reichel and T. Weth, A priori bounds and a Liouville
  theorem on a half-space for higher-order elliptic Dirichlet
  problems,  Math. Z. {\bf 261} (2009), 805--827.

\bibitem{RW2} W. Reichel and T. Weth, Existence of solutions to
  nonlinear, subcritical higher order elliptic Dirichlet problems,
  J. Differential Equations {\bf 248} (2010), 1866--1878.

\bibitem{WX} J. Wei and X. Xu, Classification of solutions of higher
  order conformally invariant equations,  Math. Ann. {\bf 313} (1999),
  207--228.

\bibitem{X}X. Xu, Uniqueness theorem for the entire positive solutions
  of biharmonic equations in $R^n$, Proc. Roy. Soc. Edinburgh Sect. A
  {\bf 130} (2000), 651--670.

\end{thebibliography}
\end{document}